\documentclass{amsart}
\usepackage{amssymb,verbatim}
\usepackage{hyperref}
\newtheorem{theorem}{Theorem}[section]

\newtheorem{corollary}[theorem]{Corollary}
\newtheorem{definition}[theorem]{Definition}
\let\olddefinition\definition
\renewcommand{\definition}{\olddefinition\normalfont}

\newtheorem{remark}[theorem]{Remark}
\let\oldremark\remark
\renewcommand{\remark}{\oldremark\normalfont}
\newtheorem{example}{Example}
\let\oldexample\example
\renewcommand{\example}{\oldexample\normalfont}

\numberwithin{equation}{section}

\def\tr{\ensuremath {\textnormal{tr}}}
\def\Q{\ensuremath {\mathbb Q}}

\def\C{\ensuremath {\mathbb C}}
\def\H{\ensuremath {\mathcal H}}

\def\R{\ensuremath {\mathbb R}}

\def\A{\ensuremath {\mathbb A}}
\def\a{\ensuremath {\mathfrak a}}
\def\L{\ensuremath {\mathcal L}}

\def\P{\ensuremath {\mathcal P}}

\def\Re{\ensuremath {\textnormal{Re}}}
\def\bs{\ensuremath {\backslash}}

\def\GL{\ensuremath {\mathrm{GL}}}

\title{Zeroes of Rankin--Selberg $L$-functions and the trace formula}
\author{Tian An Wong}

\begin{document}

\begin{abstract}
We express the contribution of certain maximal parabolic Eisenstein series to the spectral side of the Arthur--Selberg trace formula for GL$(n)$ in terms of zeroes of Rankin--Selberg $L$-functions, generalizing previous results for GL(2) and Hecke $L$-functions. As applications, we prove a lower bound for the sum of these zeroes, and a base change relation between the zeroes in the case $n=2$ for cyclic extensions of prime degree.
 \end{abstract}

\maketitle

\section{Introduction}

In \cite{W}, we introduced a new method of studying the zeroes of $L$-functions on GL$_2$ by applying the explicit formulas to the spectral side of the trace formula. In this note, we extend this method to Rankin--Selberg $L$-functions on $G=\GL_n$ by considering the noninvariant Arthur--Selberg trace formula and the Langlands--Shahidi method for maximal parabolic Eisenstein series. More precisely, let $G=\GL_n$ over $\Q$ and let $\A$ be the ring of adeles of $\Q$. Then given 
\[
G(\A)^1 =\{x\in G(\A): |\det(x)|=1\},
\]
the noninvariant trace formula is an identity of distributions, 
\[
J_\text{spec}(f) = J_\text{geom}(f), \qquad f\in C_c^\infty(G(\A)^1).
\]
We shall only be concerned with the left-hand side, which is the spectral side of the trace formula. The spectral side has a coarse expansion
\[
J_\text{spec}(f)=\sum_{\chi\in\mathfrak X}J_\chi(f)
\]
into distributions $J_\chi$ parametrized by the set of cuspidal data $\mathfrak X${,} which consists of the Weyl group orbits of pairs $(M_P,\sigma_P)$ where $M_P$ is the Levi component of a standard parabolic subgroup of $G$ and $\sigma_P$ is an irreducible cuspidal automorphic representation of $M_P(\A)^1$.  Given a Levi subgroup $M$, let $\L(M)$ be the set of Levi subgroups of $G$ containing $M$. Let us also fix a minimal Levi $M_0$, and write $\L=\L(M_0)$. Also let $\P(M)$ be the set of standard parabolic subgroups of $G$ with Levi component equal to $M$.
Then it follows from M\"uller and Speh \cite{MS} that we can express $J_\chi(f)$ as the absolutely convergent expression
\[
J_\chi(f)=\sum_{M\in \mathcal L}\sum_{L\in \mathcal L(M)}\sum_{P\in \mathcal P(M)}\sum_{s\in W^L(\a_M)_\text{reg}}a^L_{M}(s)J^L_{M,P}(f,s),
\]
where $W^L(\a_M)_\text{reg}$ is a certain subset of the Weyl group, $a^L_{M}(s)$ is a global coefficient, and $J^L_{M,P}(f,s)$ is a sum of integrals of weighted characters which we recall in \eqref{Jspec}.

Our main interest in this paper is in the contribution of the terms corresponding to $M=L$ with corank 1 in $G$, for which we derive an expression for the spectral contribution in terms of zeroes of Rankin--Selberg $L$-functions. To state the main result, we have to first introduce some notation. Let $A_P$ be the split component of the center of $M_P$ with Lie algebra $\a_M$, and let $\a^G_M$ be the kernel of the natural map from $\a_M$ to $\a_G$. Let $\alpha$ be the unique simple root of $(P,A_P)$, and let $\varpi$ be the element in $(\a^G_M)^*$ such that $\varpi(\alpha^\vee)=1$. In this case, the contribution to $J_\chi(f)$ is equal to the sum over the set of equivalence classes of irreducible unitary representations $\Pi(M(\A)^1)$ of $M(\A)^1$ of 
\begin{equation}
\label{ML1}
-a^G_M\int_{i\R}\tr(M_{\bar{P}|P}(z\varpi)^{-1}M'_{\bar{P}|P}(z\varpi)\rho_{\chi,\pi}(P, z\varpi,f))d|z|,
\end{equation}
where $a^G_M$ is the coefficient defined in \eqref{agm}, $M_{\bar{P}|P}$ is a certain intertwining operator and $\rho_{\chi,\pi}(P, z\varpi)$ is an induced representation defined in Section \ref{fse}. We write $h_{\chi,\pi}(z)=h_\pi(z)$ for the trace of the operator $\rho_{\chi,\pi}(P, z\varpi,f)$, and $\hat{h}_{\chi,\pi}(z)=\hat{h}_\pi(z)$ for its Mellin transform.

Given a pair of standard parabolic subgroups $P,P'$, we define the Weyl set $W(\a_P,\a_{P'})$ as the set of
distinct linear isomorphisms from $\a_P\subset \a_0$ onto $\a_P'\subset \a_0$ obtained by restriction
of elements in the Weyl group $W_0^G$ of $G$. We say that a class $\chi\in\mathfrak X$ is unramified if for every pair $(M_P,\sigma_P)$ in $\chi$, the stabilizer $W(\a_P,\a_P)$ is trivial, and we call $\chi$ ramified otherwise. 

\begin{theorem}
\label{ML}
Let $\chi$ be an unramifed cuspidal automorphic datum associated to the maximal Levi $M\simeq \GL_{n_1}\times \GL_{n_2}$, and let $\pi_i$ be {unramified} cuspidal automorphic representations of $\GL_{n_i}$ for $i=1,2$. Then \eqref{ML1} is equal to the sum over $\Pi(M(\A)^1)$ of the product of $a^G_M$ and
\begin{align}
\label{bigsum}&\sum_{\rho}2\pi h_\pi(\rho)-2\pi \delta(\pi_1,\pi_2) - \int_{i\R}\frac{L'}{L}(1+z,\pi_1\times\tilde\pi_2)(h_\pi(z) + h_\pi(1+z))dz\\
&- 2\hat{h}_\pi(1)\log A - \sum_{{p}\in S}\int_{i\R}N_{\bar P|P}({\pi_p},z\varpi)^{-1}\frac{d}{dz}N_{\bar P|P}({\pi_p},z\varpi)h_\pi(z)d|z|\notag
\end{align}
where $\rho$ runs over nontrivial zeroes of $L({z},\pi_1\times\tilde{\pi}_2)$. If $\chi$ is ramified, the same formula holds except with the additional term \eqref{ram}.
\end{theorem}

While we have worked  specifically with the group $\GL_n$, which leads us to Rankin--Selberg $L$-functions that are of particular interest to analytic number theory, we note that the generalization of these results to a maximal Levi in a general reductive group is largely formal, in which case we encounter other automorphic $L$-functions that arise from the Langlands--Shahidi method. 

Let us write the expression that we have obtained in \eqref{bigsum} as
\[
\sum_{\rho}2\pi h_\pi(\rho) - D_{\pi_1,\pi_2}(f),
\]
where we continue to suppress the dependence on $\chi$. Then Theorem \ref{ML} can be expressed more succinctly as 
\[
a^M_{M}(s)J^M_{M,P}(f,s) = \sum_{\pi\in\Pi(M(\A)^1)} a^G_M\left(\sum_{\rho}2\pi h_\pi(\rho) - D_{\pi_1,\pi_2}(f)\right).
\]
Here again $\rho$ runs over the nontrivial zeroes of $L(s,\pi_1\times\tilde\pi_2)$, where by nontrivial we shall always means zeroes in the critical strip, and $D_{\pi_1,\pi_2}$ is a distribution on $C_c^\infty(G(\A)^1)$ given in explicit terms.

We provide two applications of this method; the first is a generalization of the main theorem of \cite{W}, giving a lower bound on the sum of zeroes of Rankin--Selberg $L$-functions. For any $f_0\in C_c^\infty(G(\A)^1)$, let us denote $f^*(x)=\overline{f_0(x^{-1})}$.

\begin{corollary}
\label{weil}
For any $f= f_0*f_0^*$ with $f_0\in C_c^\infty(G(\A)^1)$ and $\pi\in \Pi(M(\A)^1)$, we have 
\[
\sum_\rho h_\pi(\rho) \ge  D_{\pi_1,\pi_2} (f)
\]
where $D_{\pi_1,\pi_2}(f)$ is the distribution defined by the terms above.
\end{corollary}

\noindent This falls short of giving a necessary and sufficient condition for the Riemann hypothesis for $L(s,\pi_1\times\tilde\pi_2)$, parallel to the Weil criterion for the Riemann hypothesis \cite{W72}, for the reason that our Fourier transforms are nonabelian, and the usual proofs of the criterion require explicit Fourier transforms of carefully chosen test functions, which are not available in the general nonabelian setting.

As a second application, we deduce from Langlands' proof of base change for GL$_2$ a functorial relation between the zeroes of Hecke $L$-functions attached to a cyclic extension $E$ of $\Q$ of prime degree, with those of Hecke $L$-functions over $\Q$. Given a function $f\in C_c^\infty(G(\A_E))$, we write $f'\in C_c^\infty(G(\A))$ for the base change transfer of $f$. We shall write $h'_\pi$ for the corresponding character $\tr(\rho_{\chi,\pi}(P,z\varpi,f')$, and $\tilde{h}_\pi$ and $\tilde{D}_{\pi_1,\pi_2}$ for the twisted analogues with respect to a fixed element $\sigma$ in Gal$(E/\Q)$.
\begin{corollary}
\label{bc}
Let $M$ be the set of diagonal matrices of $\GL_2$, and $\eta_E=(\mu_E,\nu_E)$ a unitary character of $M(E)\bs M(\A_E)$.  Then we have
\[
\sum_{\eta\mapsto \eta_E} \sum_{\rho_\eta} \left(h'_\eta(\rho_\eta) - D_{\mu,\nu}(f')\right) = \sum_{\rho_{\eta_E}}\left( \tilde{h}_{\eta_E}(\rho_{\eta_E}) - \tilde{D}_{\mu_E,\nu_E}(f)\right)
\]
where the first sum is over characters $\eta=(\mu,\nu)$ such that  $\eta_E = \eta \circ N_{E/\Q}$, and the other sums run over nontrivial zeroes $\rho$ and $\rho_E$ of the corresponding {$L$-functions} $L(s,\chi)$ and $L(s,\chi_E)$, counted with multiplicity.
\end{corollary}

Such a relation would of course be expected by the general principle of functoriality, but we are not aware of results of this nature in the literature. Our present methods are limited to the case of $n=2$ as the continuous spectral terms are no longer treated in such an explicit form in higher rank, though the identity that we obtain is certainly suggestive of other such relations.

\subsection{Relation to other work} We note that a crucial difference between our method and other approaches to explicit formula in higher rank is that we encounter {\em nonabelian} Fourier transforms of compactly-supported test functions on $G$. This naturally arises as we work from the standpoint of the trace formula, so that our methods rely on invariant harmonic analysis. Also, we remark that Shahidi \cite{Sha} has introduced a new normalisation of the intertwining operator, which may lead to a simplification of the terms in the main theorem.

Explicit formulas for $L$-functions on $\GL_n$ are now a cornerstone in the analytic theory of automorphic forms. They are widely used for example to study low-lying zeroes of families of $L$-functions, beginning with the groundbreaking work of Iwaniec, Luo, and Sarnak \cite{ILS}. More recently, from an algebraic point of view, Chenevier and Lannes \cite[Theorem 9.3.2]{CL} were able to prove using the explicit formula for Rankin--Selberg $L$-functions the finiteness of algebraic cuspidal automorphic representations of PGL$_n$ of motivic weight $w\le  22$. We also note the recent works \cite{BTZ,BH} related to zeroes of Rankin--Selberg $L$-functions.

\subsection{Summary} This paper is organized as follows. In Section \ref{TF} we recall the spectral expansion of the Arthur--Selberg trace formula. In Section \ref{RS} we recall the properties of Rankin--Selberg $L$-functions that we require, and apply the residue method to produce the sums over zeroes in the spectral expansion and the proof of Theorem \ref{ML}. In Section \ref{app} we then deduce Corollaries \ref{weil} and \ref{bc} from the main theorem.

\section{The spectral side of the trace formula}
\label{TF}
\subsection{Preliminaries} 

Fix $G=\GL_n$. For simplicity, we shall work over $\Q$, though most of our analysis is valid for number fields, since we work adelically. Let $\A=\A_\Q$ denote the adele ring of $\Q$. Let $G(\A)^1$ be the set of $x\in G(\A)$ such that $|\det(x)|=1$. Denote by $\rho(g)$ the representation of $G$ on $L^2(G(\Q)\backslash G(\A)^1)$ acting by right translation. Given any  $f\in C^\infty_c(G)$, one defines the integral operator
\[\rho(f)=\int_Gf(x)\rho(x)dx\]
acting on the $L^2$-space. Note that for functions of the form 
\[
f(x)=f_0(x)*f_0^*(x)=f_0(x)*\overline{f_0(x^{-1})}, \qquad f_0\in C^\infty_c(G)
\]
where $*$ denotes the usual convolution, the operator is self-adjoint and positive definite \cite[2.4]{GGPS}, and as a consequence its restriction to any invariant subspace is also positive definite. We note that if $\pi$ is a {subrepresentation of $\rho$}, one sees that the identity
\[
\pi(f) = \pi(f_0*f_0^*) = \pi(f_0) \pi(f_0^*)
\]
holds.
By the theory of Eisenstein series, $L^2(G(\Q)\backslash G(\A)^1)$ decomposes under the action of $\rho$ into discrete and continuous parts, and the discrete spectrum is made out of cuspidal and one-dimensional subspaces:
\[L^2_\text{disc}(G)\oplus L^2_\text{cont}(G)=L^2_\text{cusp}(G)\oplus L^2_\text{res}(G)\oplus L^2_\text{cont}(G)\]
where for short we have written $L^2(G)$ for $L^2(G(\Q)\backslash G(\A)^1)$, and the cuspidal spectrum consist of cusp forms. The continuous spectrum is described by the inner product of Eisenstein series, and the residual spectrum described by the residues of Eisenstein series. An explicit description of these is given in the work of {Mœglin} and Waldspurger \cite{MW}. 

We shall recall some definitions in order to describe the spectral side of Arthur's noninvariant trace formula, referring to Part 1 of \cite{A} for details. Let $P$ be a standard parabolic subgroup of $G$ with unipotent radical $N_P$ and Levi component $M\supset M_0$. Let $A_P$ the split component of $M_P$, with Lie algebra $\a_M\simeq\R^r$. That is, if $X(M)$ is the group of characters of $M$ over $\Q$, then
\[
\a_M=\text{Hom}(X(M)_\Q,\R),\quad\a^*_M=X(M)_\Q\otimes\R.
\]
The Weyl group $W(\a_M)$ coincides with the symmetric group on $r$ letters. Let $A_P(\R)^0$ be the connected component of the identity of $A_P(\R)$, and let $K$ be a maximal compact subgroup of $G(\Q)$. Define $\H_P$ to be the Hilbert space of measurable functions
\[
\phi: N_P(\A) M_P(\Q)A_P(\R)^0\bs G(\A) \to \C
\]
such that for any $x\in G(\A)$ the function $\phi_x(m)=\phi(mx)$ belongs to the space $L^2_\text{disc}(M_P(\Q)\bs M_P(\A)^1)$ and such that
\[
\int_K\int_{M_P(\Q)\bs M_P(\A)^1}|\phi(mk)|^2dm\ dk
\]
is finite. Here $M_P(\A)^1 = M_P(\A)\cap G(\A)^1$. Then given $x\in G(\A)$, $\phi\in \H_P$, and $\lambda\in \a^*_{M,\C} = \a^*_M\otimes \C$, we define the Eisenstein series
\[
E(x,\phi,\lambda) = \sum_{\delta\in P(\Q)\bs G(\Q)} \phi(\delta x) e^{(\lambda + \rho_P)(H_P(\delta x))}
\]
where $e^{\rho_P(H_P(\cdot))}$ is the square root of the modular function on $P$ \cite[\S7]{A}. It converges absolutely for Re$(\lambda)\gg0$. 

Let $P$ be a standard parabolic subgroup of $G$ with Levi component $M_P$, and $\sigma$ an irreducible cuspidal automorphic representation of $M_P(\A)^1$. We call the set of $\chi=(P,\sigma)$ cuspidal automorphic data, determined up to conjugacy, denoted $\mathfrak X$. There is an orthogonal decomposition of the discrete spectrum
\[
L^2_\text{disc}(G(\Q)\backslash G(\A)^1)=\bigoplus_{\chi\in\mathfrak X} L^2_{\text{disc},\chi}(G(\Q)\backslash G(\A)^1)
\]
into $G(\A)$-invariant subspaces. Define $\rho_{\chi}(\sigma,\lambda)$ to be the induced representation acting on functions $\phi$ on $G(\A)$ whose restriction to $M_P(\A)^1$ belongs to $L^2_\chi(M_P(\Q)\bs M_P(\A)^1)$ consisting of functions such that 
\[(\rho_{\chi}(\sigma,\lambda,y)\phi)(x)=\phi(xy)e^{(\lambda+\rho_P)H_P(xy)}e^{-(\lambda+\rho_P)(H_P(x))}\]
which defines the integral operator
\[(\rho_{\chi}(\sigma,\lambda,f)\phi)(x)=\int_{G(\A)}f(y)(\rho_{\chi}(\sigma,\lambda,y)\phi)(x)dy\]
for any $f\in C_c^\infty(G(\A)^1)$ and $\phi$ as above. We note that the trace formula is in fact valid for a larger class of noncompactly supported test functions $\mathcal C^1(G(\A)^1)$ by the results of M\"uller and Speh \cite{MS} in the case of GL$_n$, which we do not need here.

The general formula for the spectral side is the absolutely convergent sum
\[
J(f) = \sum_{\chi\in\mathfrak X} J_\chi(f),\quad f\in C_c^\infty(G(\A)^1).
\]

\subsection{The fine spectral expansion} 
\label{fse}
We now recall the fine spectral expansion according to \cite{MS}. For any pair of Levi subgroups $M\subset L$, there is a natural surjection from $\mathfrak a_M$ onto $\mathfrak a_L$ with kernel denoted by $\mathfrak a^L_M$.   Let $\Pi_\text{disc}(M(\A)^1)$ be the set of $\pi\in\Pi(M(\A)^1)$ which are equivalent to an irreducible subrepresentation of the regular representation of $M(\A)^1$ in $L^2(M(\Q)\backslash M(\A)^1)$. Then the spectral expansion can be written as
\begin{equation}
\label{Jspec}
J_\chi(f)=\sum_{M\in \mathcal L}\sum_{L\in \mathcal L(M)}\sum_s \sum_{P\in \mathcal P(M)}a^L_{M}(s)J^L_{M,P}(f,s)
\end{equation}
where sum over $s$ is taken over elements in 
\[
W^L(\mathfrak a_M)_\text{reg}=\{t\in W^L(\mathfrak a_M):\ker(1-t)=\mathfrak a_L\},
\]
where $W^L(\mathfrak a_M)$ is the Weyl group of $A_M$ relative to $L$. Also, we have the global coefficient
\begin{equation}
\label{alms}
a^L_M(s)={|W_0^M|}{|W^G_0|^{-1}|\P(M)|^{-1}}|\det(s-1)_{\mathfrak a^L_M}|^{-1},
\end{equation}
and the sum of integrals
\[
J^L_{M,P}(f,s)=\sum_{\pi\in\Pi_\text{disc}(M(\A)^1)}\int_{i\mathfrak a^*_L/\mathfrak a^*_G}\tr(\mathfrak M_L(P,\lambda)M_{P|P}(s,0)\rho_{\chi,\pi}(P,\lambda, f))d \lambda.
\]
Here $M_{P|P}(s,0)$ is the global intertwining operator $M_{Q|P}(s,\lambda)$ at $P=Q$ and $\lambda\in i\a_L^*$ which for Re$(\lambda)$ in a certain chamber, can be defined by an absolutely convergent integral and admits an analytic continuation to a meromorphic function of $\lambda$. Also, $\mathfrak M_L(P,\lambda)$ is given by its restriction to $\pi$,
\[\mathfrak M_L(P,\pi,\lambda)=\sum_S \mathfrak N'_S(P,\pi,\lambda)\nu^S_L(P,\pi,\lambda)\]
where $S$ runs over parabolic subgroups containing $L$, $\mathfrak N'_S(P,\pi,\lambda)$ is built out of normalized intertwining operators on the local groups $G(\Q_p)$ and $\nu^S_L(P,\pi,\lambda)$ is a scalar valued function which is defined in terms of normalizing factors, the general formulae for which we refer to \cite{arteis2}. Finally, $\rho_{\chi,\pi}(P,\lambda)$ is an induced representation with complex parameter $\lambda,$ and acts as the regular representation on the subspace of automorphic forms $\phi$ where $\phi_x(m)=\phi(mx)$ belongs to the  {$\pi$-isotypic} subspace of $L^2(M_P(\Q)\backslash M_P(\A)).$

The important point is that the global intertwining operator is built out of normalized intertwining operators 
\[
M_{Q|P}(\pi,\lambda)=r_{Q|P}(\pi,\lambda)N_{Q|P}(\pi,\lambda).
\]
Consider the following terms in the expansion \eqref{Jspec}. 
\begin{enumerate}
\item[(a)]
$M=L\neq G$, and $\dim\mathfrak a^G_L=1$. In this case, $|\P(M)|=1$, $\a_M^L$ is trivial, and the operator $\mathfrak M_L(P,\lambda)$ has a simple description. Let $\alpha$ be the unique simple root of $P$, $\varpi$ be the element in  $(\mathfrak a^G_L)^*=\R$ such that $\varpi(\alpha^\vee)=1$, and $\bar P$ the opposite parabolic of $P$. {Also denote by} $\text{vol}(\mathfrak a^{G}_M/\mathbb Z(\alpha^\vee))$ is the volume of the lattice in $\a^G_M$ spanned by $\alpha^\vee$. Then setting $\lambda = z\varpi, z\in i\R$, we have
\[
\mathfrak M_L(P,\lambda) = -\text{vol}(\mathfrak a^G_M/\mathbb Z\alpha^\vee)M_{\bar{P}|P}(z\varpi)^{-1}M'_{\bar{P}|P}(z\varpi).
\]
It follows then that the term $a^M_{M}(s)J^M_{M,P}(f,s)$ is equal to the sum over $\pi\in\Pi(M(\A)^1)$ of 
\[
-{|W^M_0|}{|W^G_0|^{-1}}\text{vol}(\mathfrak a^G_M/\mathbb Z\alpha^\vee)
\int_{i\R}\tr(M_{\bar{P}|P}(z\varpi)^{-1}M'_{\bar{P}|P}(z\varpi)\rho_{\chi,\pi}(P, z\varpi,f))d|z|
\]
where $M_{\bar{P}|P}^{-1}M'_{\bar{P}|P}$ separates into
\begin{equation}
\label{logd}
r_{P|\bar P}(\pi,z)^{-1}r'_{P|\bar P}(\pi,z)+\sum_{{p}\in S}N_{\bar P|P}({\pi_p},z\varpi)^{-1}\frac{d}{dz}N_{\bar P|P}({\pi_p},z\varpi).
\end{equation}
Here $N_{\bar{P}|P}$ is the local intertwining operator, and $S$ runs over a finite set of places outside of which $N_{\bar{P}|P}$ is unramified{, and $r_{P|\bar P}(\pi,z)$ is a normalizaing factor of the intertwining operator, defined as in \eqref{rPP} below}.

\item[(b)]
$M \neq L \neq G$, and $\dim\a^G_L=1$. In this case $M$ is no longer maximal, and the contribution to \eqref{Jspec} is then given by the sum over $s\in W^L(\mathfrak a_M)_\text{reg}$, $\pi\in\Pi(M(\A)^1)$, and $P\in \P(M)$ of the term
\[
a^G_{M}(s)J^G_{M,P}(f,s)
\]
which is equal to
\[-a^L_M(s)\text{vol}(\mathfrak a^G_L/\mathbb Z\alpha^\vee)
\int_{i\R}\tr(M_{\overline{P}|P}(z\varpi)^{-1}M'_{\overline{P}|P}(z\varpi)M_P(s,0)\rho_{\chi,\pi}(P, z\varpi,f))d|z|,
\]
where $M_{\bar{P}|P}^{-1}M'_{\bar{P}|P}$ again separates as above. We do not treat this case in this paper, but we note that it is the next simplest case.
\end{enumerate}

In particular, these two cases represents the terms on the spectral side containing one-dimensional integrals.

\section{Zeroes of Rankin--Selberg $L$-functions}
\label{RS}
\subsection{Rankin--Selberg $L$-functions}

To fix notation, we briefly review Rankin--Selberg $L$-functions. Let $P$ be a standard maximal parabolic subgroup with Levi component $M\simeq \GL_{n_1}\times \GL_{n_2}$. Then for each $\pi_i\in\Pi(\GL_{n_i}), i=2$ we associate the Rankin--Selberg $L$-function  $L({z},\pi_1,\times\tilde\pi_2)$. It satisfies the functional equation
\[L({z},\pi_1\times\tilde{\pi_2})=\epsilon({z},\pi_1\times\tilde{\pi_2})L(1-{z},\tilde\pi_1\times{\pi_2})\]
in which
\[\epsilon({z},\pi_1\times\tilde{\pi_2})=W(\pi_1,\tilde\pi_2)A^{\frac{1}{2}-{z}}\]
where $W(\pi_1,\tilde\pi_2)$ is the root number {of absolute value 1} satisfying $W(\pi_1,\tilde\pi_2)W(\tilde\pi_1,\pi_2)=1$, and {
\[
A = \prod_{p} A_p,\quad A_p=p^{f_p(\pi_1\times\tilde\pi_2)}.
\]
Here $f_p(\pi_1\times\tilde\pi_2)$ is the local conductor exponent, a nonnegative} integer such that $f_p(\pi_1\times\tilde\pi_2) = f_p(\tilde\pi_1\times\pi_2)$.

It is well-known by \cite{JS,JPSS} that given cuspidal representations $\pi_i$ of $\GL_{n_i}$, then $L({z},\pi_1\times\tilde\pi_2)$ is entire unless $\pi_1\simeq \pi_2\otimes|\cdot|^w$ for some $w$ in $\C$, in which case it is holomorphic except for simple poles at ${z}=1-w$ and $-w$. Moreover, it can be shown that the $L({z},\pi_1\times\tilde\pi_2)$ converges in the half plane Re$({z})>1$.  The normalizing factor in this case is then given by the quotient
\begin{equation}
\label{rPP}
r_{\bar{P}|P}(\pi_1{\times} \tilde\pi_2,{z}) 
= \frac{L({z},\pi_1\times\tilde{\pi_2})}{L(1+{z},\pi_1\times\tilde{\pi_2})\epsilon({z},\pi_1\times\tilde\pi_2)}
\end{equation}
and is a meromorphic function of ${z}$.

\subsection{Explicit formulas}

We now consider the contribution to the trace formula arising from the first case (a) $M=L\neq G$, in which $M$ is the Levi component of a maximal parabolic subgroup, hence $\dim\mathfrak a^G_L=1$. For ease of notation let us write
\begin{equation}
\label{agm}
a^G_M = |W^M_0||W^G_0|^{-1}\mathrm{vol}(\R/\mathbb Z\alpha^\vee). 
\end{equation}
If $\chi$ is unramified, $J_\chi(f)$ is equal to the sum 
\[
- \sum_{\pi\in\Pi(M(\A)^1)}a^G_M\int_{i\R}\tr(M_{\bar{P}|P}(z\varpi)^{-1}M'_{\bar{P}|P}(z\varpi)\rho_{\chi,\pi}(P, z\varpi,f))d|z|
\]
where we recall that $M_{\bar{P}|P}^{-1}M'_{\bar{P}|P}$ separates according to \eqref{logd} and that the integral is absolutely convergent. If $\chi$ is ramified, there is the additional discrete term
\begin{equation}
\label{ram}
\sum_{s\in W^L(\mathfrak a_M)_\text{reg}}\sum_{\pi\in\Pi(M(\A)^1)}a^L_M(s)\tr(M_{P|P}(s,0)\rho_{\chi,\pi}(P,0,f)).
\end{equation}
It is the generalization of the term (vi) in the trace formula for GL$_2$ in \cite[\S16]{JL}.

Now we take the logarithmic derivative of $r_{P|\bar P}(\pi,{z})$ to get
\begin{equation}
\label{logdl}
\frac{L'}{L}({z},\pi_1\times\tilde\pi_2)-\frac{L'}{L}(1+{z},\pi_1\times\tilde\pi_2)-\frac{\epsilon'}{\epsilon}({z},\pi_1\times\tilde\pi_2).
\end{equation}
The first term produces the integral
\begin{equation}
\label{int1}
-a^G_M\int_{i\R}\frac{L'}{L}(z,\pi_1\times\tilde\pi_2)\tr(\rho_{\chi,\pi}(P, z\varpi,f))d|z|,
\end{equation}
where the trace $\tr(\rho_{\chi,\pi}(P, z\varpi,f))$ lies in the Paley-Wiener space of entire, Weyl-invariant functions on $\a^*_{P,\C}$ of exponential type and rapidly decreasing on cylinders \cite{CD, BDK}. For ease of notation, we denote the trace as
\[
h_\pi(z) = \tr(\rho_{\chi,\pi}(P, z\varpi,f)).
\]
Let us write its Mellin transform as 
\[
\hat{h}_\pi(z) = {\int_{0}^\infty} h_\pi(x)x^{z-1} dx,
\]
and let {
\[
\delta(\pi_1,\pi_2) = 
\begin{cases} 
h_\pi(1-w)+ h_\pi(w) & \text{if }\pi_1\simeq \pi_2\otimes |\cdot|^w, 0\le \text{Re}(w)\le 1, \\
0, &  \text{otherwise}.
\end{cases}
\]}
We now turn to the proof of the main theorem. It expresses $J_\chi(f)$ as a sum over zeroes of Rankin--Selberg $L$-functions. 

\begin{proof}[Proof of Theorem \ref{ML}]
The method of proof is essentially that of \cite[Theorem 1.1]{W}. It will suffice to consider the normalizing factor in \eqref{logd}, as we leave the local intertwining operators as they are. We shall interpret the integral
\begin{equation}
\label{rint}
\int_{i\R}\tr(r_{\bar{P}|P}(\pi,z\varpi)^{-1}r'_{\bar{P}|P}(\pi,z\varpi)\rho_{\chi,\pi}(P, z\varpi,f))d|z|
\end{equation}
as the limit of a sum of contour integrals and apply the residue theorem to each one. Consider first the contribution of \eqref{int1} to \eqref{rint}. We shall relate the integral to the limit of a contour integral around a rectangle $R$ with vertices $(\pm iT,1\pm iT)$, oriented positively. Namely,
\begin{equation}
\label{R}
\lim_{T\to\infty}\frac{1}{2\pi i}\int_R\frac{L'}{L}(z,\pi_1\times\tilde\pi_2)h_\pi(z)dz.
\end{equation}
If $L({z},\pi_1\times\tilde\pi_2)$ has a pole at ${z}=1-w$ and $-w$, then we deform the contour as necessary along an arbitrarily small semicircle to the right (resp. left) of the poles. The right boundary is then within the region of absolute convergence, and furthermore $L({z},\pi_1\times\tilde\pi_2)$ has no poles or zeroes in the region $\Re({z})>1$. 

Along the horizontal edges, we cross the critical strip while avoiding zeroes of $L({z},\pi_1\times\tilde\pi_2)$ using zero density estimates and bounds on the logarithmic derivative of the Rankin--Selberg $L$-function (see for example \cite[\S9.3.5]{CL}, which in turn relies on \cite[I.2]{mestre}). In particular, there exists constants $c_1,c_2,$ and a sequence $T_m$ for every integer $m$ with $|m|\ge2$, such that $L({z},\pi_1\times{\tilde\pi}_2)$ has no zeroes in the horizontal strips
\[
|t\pm T_m|\leq\frac{c}{\log|m|}\qquad m<T_m<m+1,
\]
and 
\[
\left|\frac{L'}{L}(\sigma+iT_m,\pi_1\times{\tilde\pi}_2)\right| \le c_1(\log|m|)^2 , \quad -1< \sigma< 2.
\]
Then letting $T$ to tend to infinity along the sequence $T_m$ on each side, together with the rapid decay of $h_\pi(z)$, the contribution from the horizontal edges also vanish. 

We are left with the contribution from the left boundary, which in the limit is equal to our integral, and the right boundary, which is 
\[
\frac{1}{2\pi i }\int_{1-i\infty}^{1+i\infty} \frac{L'}{L}(z,\pi_1\times\tilde\pi_2)h_\pi(z)dz.
\]
If $\pi_1\not\simeq\pi_2\otimes|\cdot|^w$ {for any $w\in\C$} or $\pi_1\simeq\pi_2\otimes|\cdot|^w$ with {Re}$(w)>1$, by the residue theorem the integral \eqref{R} is equal to 
\[
 \sum_{\rho}h_\pi(\rho)\]
where the sum runs over nontrivial zeroes of $L({z},\pi_1\times\tilde\pi_2)$. If {$\pi_1\simeq \pi_2\otimes |\cdot|^w$ with} $0\le {\text{Re}(w)\le }1$, then the additional terms
\[
h_\pi(1-w) + h_\pi(w)
\]
appears. It follows then that we have a total contribution of $2\pi  a^G_M$ times 
\[
\sum_\rho h_\pi(\rho) - \delta(\pi_1,\pi_2) - \frac{1}{2\pi i }\int_{1-i\infty}^{1+i\infty} \frac{L'}{L}(z,\pi_1\times\tilde\pi_2)h_\pi(z)dz
\]
from the first term in the logarithmic derivative. Arguing similarly with the second term in \eqref{logdl}, the contribution of the first two terms can be written together as
\[
- \sum_{\pi \in \Pi(M(\A)^1)}a^G_M\int_{i\R}\frac{L'}{L}(1+z,\tilde\pi_1\times\pi_2)\left(h_\pi(z)+h_\pi(1+z)\right) d|z|.
\]
Thirdly, for the epsilon factors we have simply
\[\frac{\epsilon'}{\epsilon}({z},\pi_1\times\tilde\pi_2)=\frac{\epsilon'}{\epsilon}({z},\tilde\pi_1\times\pi_2)=-\log A,
\]
which gives $- 2\hat{h}_\pi(1) \log A$.
\end{proof}

We note that the remaining integral of the $L$-function can be expressed as a sum over primes, after shifting the contour slightly into the region of absolute convergence, as is typical in the derivation of explicit formulas.  Clearly this method works well for such $L$-functions occurring in this manner, that is by the Langlands--Shahidi method and in the maximal parabolic case, but what is missing in the general setting is the precise location of the poles of the $L$-functions.

\section{Applications}
\label{app}
\subsection{Lower bounds on sums of zeroes}
\label{sums}
 
Let $\eta$ be a Hecke character of $\Q$. We recall Weil's criterion as follows: the Riemann hypothesis for the Hecke $L$-function $L({z},\eta)$ holds if and only if the sum
\[
\sum_{\rho} h_\pi(\rho) 
\]
over nontrivial zeroes of $L({z},\eta)$ is nonnegative for all $h$ that are Fourier transforms of functions $g*g^*$ where $g$ is an even function in $C_c^\infty (\R^\times)$. 
In our case, for each distribution $J_\chi(f)$ we have an explicit formula for a family of Rankin Selberg $L$-functions, in terms of distributions arising from the trace formula. We shall prove a lower bound for the sums over zeroes. From Theorem \ref{ML} we have the identity
\begin{equation}
\label{D}
a^M_{M}(s)J^M_{M,P}(f,s) = \sum_{\pi\in\Pi(M(\A)^1)} a^G_M\left(\sum_{\rho}2\pi h_\pi(\rho) - D_{\pi_1,\pi_2}(f)\right).
\end{equation}
This relation gives us an approach to the zeroes of $L$-functions using the trace formula. The key to the analysis will be the positivity derived from the truncation operator. We can then prove the first corollary of Theorem \ref{ML}. 

\begin{proof}[Proof of Corollary \ref{weil}]

Recall from \cite[\S13]{A}{,} for example, that the spectral term $J_\chi^T(f)$ can be expressed as
\[
J_\chi^T(f)=\sum_P\frac{1}{n_P}\sum_{\sigma\in\Pi(M(\A)^1)}\int_{i\a^*_P/i\a^*_G}\tr(\Omega^T_{\chi,\sigma}(P,\lambda)\rho_{\chi}(\sigma,\lambda,f))d\lambda
\]
where the first sum is over associated parabolics $P$, $n_P$ is the number of Weyl chambers in $\a_P$, the second sum is over the set $\Pi(M(\A)^1)$ of classes of irreducible unitary representations of $M(\A)^1$, and the operator
\[
(\Omega^T_{\chi,\sigma}(P,\lambda)\phi',\phi)=\int_{G(\Q)\backslash G(\A)^1}\Lambda^TE(x,\phi',\lambda)\overline{\Lambda^TE(y,\phi,\lambda)}dx
\]
is a self-adjoint, positive-definite operator, and is asymptotically equal to
\[
\omega^T_{\chi,\sigma}(P,\lambda)=\underset{\lambda=\lambda'}{\text{val}}\sum_{Q\in P(M_P)}\sum_{s\in W(\a_P)}M_{Q|P}(0,\lambda)^{-1}M_{Q|P}(s,\lambda')\frac{e^{(s\lambda'-\lambda)(Y_Q(T))}}{\theta_Q(s\lambda'-\lambda)}
\]
with error term decaying exponentially in $T$ (see \cite[\S1]{arteis} for details). Also, $\Lambda^T$ is a truncation operator defined as follows. Let $\phi$ be a locally bounded, measurable function on $G(\Q)\backslash G(\A)^1$, and $T$ a suitably regular point in $\mathfrak a_0^+$. We then define the truncation operator
\begin{equation}
\label{trunc}
(\Lambda^T\phi)(x)=\sum_P(-1)^{\text{dim}A_P/A_G}\sum_{\delta\in P_\Q\backslash G(\Q)}\int_{N_P(\Q)\backslash N_P(\A)}\phi(n\delta x)\hat{\tau}_P(H_P(\delta x)-T)dn
\end{equation}
where $\hat{\tau}_P$ is the characteristic function of the positive Weyl chamber associated to $P$, and $H_P(x)$ is the usual height function. The inner sum is finite, while the integrand is a bounded function of $n$. In particular, if $\phi$ is in $L^2_\text{cusp}(G(\Q)\backslash G(\A)^1)$ then $\Lambda^T\phi=\phi$; also $\Lambda^T E(g,\phi,s)$ is square integrable. It is self-adjoint and idempotent, hence an orthogonal projection. 

Fix a minimal Levi $M_0$ and let $P_0$ be any parabolic with Levi component $M_0$. Let $\a_0$ be the Lie algebra of the split component $A_0$ of the center of $M_0$, and let $T$ be a point in the positive chamber $\a^+_0$ in $\a_0$, suitably regular in the sense that its distance from the walls of $\a^+_0$ is large.  It is a fundamental result of Arthur that $J^T_\chi(f)$ is a polynomial in $T$, which a priori is a sufficiently regular vector in $\a^+_0${;} % in the sense that $\alpha(T)$ is large for each root $\alpha\in\Delta_0$, 
thus it extends to all $\a_0$. For our purposes, it will simplify the formula by evaluating $T$ at a distinguished point $T_0$, which for GL$_n$ is just 0. By definition, 
\[J_\chi(f)=J^{T_0}_\chi(f)\]
is therefore the constant term, and is what finally appears in the noninvariant trace formula. 
Taking $T=0$,  the positivity of $J_\chi(f)$ will give a lower bound for the sums over zeroes. 

Following the proof of the conditional convergence of $J_\chi^T$ in \cite[\S7]{A}, we see that for any positive-definite test function $f_0*f^*_0$ with $ f_0 \in C_c^\infty(G(\A)^1)$, the resulting double integral is nonnegative, and the integrals can be expressed as an increasing limit of nonnegative functions. The integral converges, and $J_\chi^T(f)$ is positive-definite. Using the absolute convergence of the spectral side, we can interchange the sum and integrals to obtain the inner product expression as above, which is positive-definite. Then restricting the positive-definite, self-adjoint operator $\Omega^T_{P,\chi}(\lambda)$ to a $\pi$-isotypic subspace, the required result follows.
\end{proof}

As mentioned in the introduction, the proof of Weil's criterion for the Riemann hypothesis requires a careful choice of test function, which appears difficult to construct explicitly in the higher rank case. On the other hand, as discussed in \cite[\S5.2]{W}, we continue to carry the parameter $T$ as it allows for a degree of flexibility that may be useful for applications.

\subsection{Base change relations} 

We give a second application which gives a functorial relation between the zeroes of automorphic $L$-functions, using the comparing trace formulas. Let $E$ be a cyclic extension of prime degree $l$ over $\Q$. We consider only the case of $\GL_2$, in which case we have Hecke $L$-functions in place of Rankin--Selberg $L$-functions. The result relies on Langlands' proof of base change for $\GL_2$, which involves explicit identities of the intertwining operators. The scalar factors are described in the comparison of terms (10.3) and (10.28) in \cite[p.200]{BC}. Let $D_E^0$ be the set unitary characters 
\[
(\mu_E|\cdot|^\frac{{z}}{2},\nu_E|\cdot|^{-\frac{{z}}{2}})
\]
of $M(E) \bs M({\A_E})$ with a fixed central character, and similar $D^0=D^0_\Q$ with measure $|d{z}|$. Let $f$ be a square integrable function on $G(E)\bs G(\A_E)$ with the same central character, and similarly $f'$ on $G(\Q)\bs G(\A_\Q)$, related by the base change homomorphism of Hecke algebras \cite[\S8]{BC},
\[
f \to f' 
\]
(we note that in \cite{BC} these are denoted as $\phi$ and $f$ respectively). We have the operator obtained from the induced representation $\rho(f,\eta)$, and if $\sigma$ is a generator of the Galois group Gal$(E/\Q)$, we have $\rho(\sigma,\eta)$. Also recall that the normalizing factor $m(\eta)$ is equal to the quotient $L(s,\nu\mu^{-1})/L(s,\mu\nu^{-1}),$ and similarly $m_E$ where the $L$-functions are defined over $E$. Let us denote $\chi = \nu\mu^{-1}$ and $\chi_E = \nu_E\mu_E^{-1}$.  Our proof of Corollary \ref{bc} then follows from Langlands' comparison of trace formulae in this context.

\begin{proof}[Proof of Corollary \ref{bc}]

Referring to \cite[\S10]{BC} for further notation, we have the relation 
\begin{align*}
&\sum_{\eta\mapsto \eta_E} \frac{1}{4\pi}\int_{D^0}m^{-1}(\eta)m'(\eta)\tr(\rho(f',\eta))|d{z}|\\
&= \frac{1}{4\pi}\int_{D^0_E}m^{-1}_E(\eta)m'_E(\eta)\tr(\rho(f,\eta)\rho(\sigma,\eta))|d{z}|,
\end{align*}
with at most $l$ terms in the sum over $\eta$. There is a surjection of $D^0$ onto $D^0_E$ induced by the norm. Thus if $\eta$ maps to $\eta_E$, we have the relation
\[
\tr(\rho(f',\eta)) = \tr(\rho(f,\eta_E)\rho(\sigma,\eta_E))
\]
and by \cite[Lemma 3.4]{W} we may view it as the Mellin transform of a compactly supported function on $\R^\times_+$. Then the integral over $D_0$ can be expressed as
\[
\sum_{\eta\mapsto \eta_E}\frac{1}{4\pi}\int_{D^0_E}  m^{-1}(\eta)m'(\eta)\tr(\rho(f',\eta_E)\rho(\sigma,\eta_E))|d{z}|
\]
and rather than cancelling terms as Langlands does, we apply the analogue of Theorem \ref{ML} (see also \cite[Theorem 1.1]{W}) to conclude. 
\end{proof}

To work in higher rank, we note that following the remark preceding \cite[Theorem 4.2]{AC}, using the results of {Mœglin} and Waldspurger \cite{MW} on the discrete spectrum of $\GL_n$, one has the extension of the weak base change lift of Arthur and Clozel to all automorphic forms appearing in the decomposition of $L^2(G(\Q)\bs G(\A))$, in particular the distributions arising from the continuous spectrum. Unfortunately, the comparisons of trace formulae in higher rank do not provide formulas as explicit as those given by Langlands for $\GL_2$. 

\subsubsection*{Acknowledgments} The author thanks Farrell Brumley for encouragement concerning this work, and for pointing out the reference \cite{CL}. {The author was partially supported by NSF grant DMS-2212924.}

\bibliographystyle{alpha}
\bibliography{master}
\end{document}